\documentclass[12pt]{article}
\usepackage{amsmath}
\usepackage[affil-it]{authblk}
\usepackage{amsfonts}
\usepackage{amsthm}
\usepackage{amssymb}
\usepackage{makecell}
\usepackage{appendix}
\usepackage{algorithmic}
\usepackage{graphicx}
\usepackage{algorithm2e}
\usepackage{tabu}
\usepackage{color}
\definecolor{darkgreen}{rgb}{0.0, 0.63, 0.0}
\theoremstyle{definition}

\usepackage{mathtools}

\theoremstyle{lemma}
\newtheorem{theorem}{Theorem}

\newtheorem{lemma}[theorem]{Lemma}

\DeclareMathOperator{\Log}{Log}

\newcommand{\vol}{\text{Vol}}
\newcommand{\norm}{\text{Nm}}

\newcommand{\trace}{\text{Tr}}

\usepackage[margin=1in]{geometry}

\begin{document}
\title{An Upper Bound on the Number of Classes of Perfect Unary Forms in Totally Real Number Fields}
\author{Christian Porter\footnote{Imperial College London, United Kingdom. Corresponding author, c.porter17@imperial.ac.uk}, Andrew Mendelsohn\footnote{Imperial College London, United Kingdom. andrew.mendelsohn18@imperial.ac.uk}}
\maketitle
\begin{abstract}
    Let $K$ be a totally real number field of degree $n$ over $\mathbb{Q}$, with discriminant and regulator $\Delta_K, R_K$ respectively. In this paper, using a similar method to van Woerden, we prove that the number of classes of perfect unary forms, up to equivalence and scaling, can be bounded above by $O( \Delta_K \exp(2n \log(n)+f(n,R_K)))$, where $f(n,R_K)$ is a finite value, satisfying $f(n,R_K)=\frac{\sqrt{n-1}}{2}R_K^{\frac{1}{n-1}}+\frac{4}{n-1}\log(\sqrt{|\Delta_K|})^2$ if $n \leq 11$. Moreover, if $K$ is a unit reducible field, the number of classes of perfect unary forms is bound above by $O( \Delta_K \exp(2n \log(n)))$.
\end{abstract}
\section{Introduction}
Let $K$ be a totally real algebraic number field of degree $n$ over $\mathbb{Q}$, with ring of integers $\mathcal{O}_K$. We associate to $K$ the embeddings $\sigma_1,\sigma_2,\dots,\sigma_n$ into $\mathbb{R}$, that necessarily fix $\mathbb{Q}$. A quadratic form $f:K^m \to K$ is defined by
\begin{align*}
f(x_1,\dots,x_m)=\sum_{k,l=1}^m f_{kl}x_kx_l.
\end{align*}
We say that $f$ is positive-definite if
\begin{align*}
\sigma_i(f)(x_1,\dots,x_m)=\sum_{k,l=1}^m \sigma_i(f_{k,l})x_kx_l
\end{align*}
is positive definite for each $i=1,\dots,n$.

Quadratic forms of the form $ax^2$, where $a \in K$, are called unary forms. We say that $a$ is a totally positive element of $K$ if the form $ax^2$ is positive definite, and we denote $K_{>>0}$ the set of totally positive elements of $K$. Note that for all $a \in K_{>>0}$, $\trace_{K/\mathbb{Q}}(ax^2)$ corresponds to a positive-definite quadratic form in $n$ variables with rational coefficients. This is known as the trace-form of the unary form $ax^2$. For each $a \in K_{>>0}$, we will use the notation
\begin{align*}
\mu(a)=\min_{x \in \mathcal{O}_K \setminus \{0\}} \trace_{K/\mathbb{Q}}(ax^2)
\end{align*}
and
\begin{align*}
\mathcal{M}(a)=\{x \in \mathcal{O}_K: \trace_{K/\mathbb{Q}}(ax^2)=\mu(a)\}.
\end{align*}
If $a^\prime= au^2$ for some $u \in \mathcal{O}_K^\times$, we say that the unary form $a^\prime x^2$ is equivalent to the form $ax^2$, and we denote this relation by $a \sim a^\prime$. Note that since the action of $\mathcal{O}_K^\times$ fixes $\mathcal{O}_K$, we have $\mu(a)=\mu(a^\prime)$, $|\mathcal{M}(a)|=|\mathcal{M}(a^\prime)|$ and each element of $\mathcal{M}(a)$ is equivalent to an element of $\mathcal{M}(a^\prime)$, and vice versa. Finally, we say that $K$ is a $A$-reducible if for every $a \in K_{>>0}$, if $x \in \mathcal{M}(a)$ then $|\norm_{K/\mathbb{Q}}(x)| \leq A$. In particular, if $A=1$ we say that $K$ is a unit reducible field.

A unary form $ax^2$ is said to be perfect if it is uniquely determined by $\mu(a)$ and $\mathcal{M}(a)$. It is immediately clear that if $ax^2$ is a perfect unary form, then any unary form of the type $\lambda a^\prime x^2$ is a perfect unary form, where $\lambda \in \mathbb{Q}^+$ and $a^\prime \sim a$. In this way, we may classify perfect unary forms up to scaling and equivalence, or in other words, by their homothety classes. It is known that for number fields of fixed dimension, the number of classes of perfect forms can grow arbitrarily large (see e.g. \cite{Y}, in which it is shown that the number of classes of perfect unary forms for real quadratic fields can become arbitrarily large).

In this paper, we follow a similar argument to van Woerden in \cite{V} in order to determine an upper bound on the number of homothety classes of perfect unary forms for an arbitrary totally real number field. Our result is stated as follows.
\begin{theorem}\label{thm1}
Let $K$ be a totally real number field with field discriminant $\Delta_K$, regulator $R_K$ and degree $n$ over $\mathbb{Q}$. Let $n_K$ denote the number of homothety classes of perfect unary forms in $K$. Then
\begin{align*}
n_K \leq e^{\sqrt{\eta_K^2+\theta_K}} |\Delta_K| \left(\frac{2}{\pi}\right)^{2n} \Gamma\left(2+\frac{n}{2}\right)^4,
\end{align*}
where
\begin{itemize}
    \item $\eta_K=0$ if $K$ is a unit reducible field,
    \item $\eta_K=\frac{\sqrt{n-1}}{2} R_K^{\frac{1}{n}}$ if $K$ is not unit reducible and $2 \leq n \leq 11$,
    \item $\eta_K=\frac{\sqrt{n-1}}{2} \left(\frac{2}{\pi}\right)^{n-1}\Gamma\left(2+\frac{n-1}{2}\right)^2\left(\sqrt{\frac{2}{n}}\left(\frac{1}{1000}\left(\frac{\log \log n}{\log n}\right)^3\right)\right)^{2-n}R_K$ otherwise,
\end{itemize}
and
\begin{itemize}
\item $\theta_K=\frac{4\log(A)^2}{n-1}$, if $K$ is $A$-reducible.
\end{itemize}

\end{theorem}
We will also prove that any totally real field $K$ is $A$-reducible, where
\begin{align*}
A \leq n^{-\frac{n}{2}} \sqrt{|\Delta_K|} \left(\frac{2}{\pi}\right)^{\frac{n}{2}}\Gamma\left(2+\frac{n}{2}\right).
\end{align*}
This leads us to the following more general theorem.
\begin{theorem}\label{thm2}
Let $K$ be a totally real number field with field discriminant $\Delta_K$, regulator $R_K$ and degree $n$ over $\mathbb{Q}$. Let $n_K$ denote the number of homothety classes of perfect unary forms in $K$. Then
\begin{align*}
n_K \leq e^{\sqrt{\eta_K^2+\rho_K}}|\Delta_K| \left(\frac{2}{\pi}\right)^{2n} \Gamma\left(2+\frac{n}{2}\right)^4,
\end{align*}
where $\eta_K$ is as defined before, and
\begin{itemize}
\item $\rho_K=0$ if $K$ is unit reducible,
\item $\rho_K=\frac{4}{n-1}\log\left(n^{-\frac{n}{2}} \sqrt{|\Delta_K|} \left(\frac{2}{\pi}\right)^{\frac{n}{2}}\Gamma\left(2+\frac{n}{2}\right)\right)^2 \approx \frac{4}{n-1}\log(\sqrt{|\Delta_K|})^2$ otherwise.
\end{itemize}
\end{theorem}
\section{Proof of Theorem \ref{thm1}}
We will define the set $K_{>>0}^{1}$ as the subset of elements in $K_{>>0}$ such that $\mu(a) = 1$. Moreover, denote $\mathcal{F}_K$ the fundamental region of $K_{>>0}$, that is, the set of elements such that for all $a \in K_{>>0}$, there exists an $a^\prime \in \mathcal{F}_K$ such that $a \sim a^\prime$. For any perfect unary form $ax^2$, we will use the notation
\begin{align*}
\mathcal{V}(a)=\left\{\sum_{x \in \mathcal{M}(a)} \lambda_i x_i^2: \lambda_i \in \mathbb{Q}^+\right\},
\end{align*}
which is called the Voronoi cone of $a$. Denote $P_K$ the set of $a \in K_{>>0}$ such that $ax^2$ determine all the distinct homothety classes of perfect unary forms. It was proven by the work of Koecher \cite{K} that
\begin{align*}
\mathcal{F}_K=\bigcup_{a \in P_K} \mathcal{V}(a),
\end{align*}
From now on, we will use the notation $\mathcal{V}^1(a)=\mathcal{V}(a) \cap K_{>>0}^1$ for any perfect unary form $ax^2$. First, we want to prove the following useful lemma.
\begin{lemma}
Suppose that $ax^2$ and $bx^2$ are perfect unary forms in $K$, and $a,b$ belong to distinct homothety classes. Then $\text{Int}(\mathcal{V}^1(a)) \cap \text{Int}(\mathcal{V}^1(b))=\phi$, where Int means the interior of the cone generated by $\mathcal{V}(a)$ or $\mathcal{V}(b)$.
\end{lemma}
\begin{proof}
Suppose that $c \in \text{Int}(\mathcal{V}^1(a)) \cap \text{Int}(\mathcal{V}^1(b))$, and $c \neq 0$. We want to show that if this holds, then $a=b$, which is a contradiction. Since $c \in \text{Int}(\mathcal{V}^1(a))$, there exist $\lambda_x \in \mathbb{Q}^+$ such that
\begin{align*}
c=\sum_{x \in \mathcal{M}(a)} \lambda_x x^2,
\end{align*}
so
\begin{align*}
\trace_{K/\mathbb{Q}}(bc)=\sum_{x \in \mathcal{M}(a)} \lambda_x \trace_{K/\mathbb{Q}}(bx^2) \geq \sum_{x \in \mathcal{M}(a)} \lambda_x=\sum_{x \in \mathcal{M}(a)}\lambda_x \trace_{K/\mathbb{Q}}(ax^2)=\trace_{K/\mathbb{Q}}(ac).
\end{align*}
Similarly, since $c \in \text{Int}(\mathcal{V}(b))$ we can deduce that $\trace_{K/\mathbb{Q}}(bc) \leq \trace_{K/\mathbb{Q}}(ac)$, so it must hold that $\trace_{K/\mathbb{Q}}(ac)=\trace_{K/\mathbb{Q}}(bc)$. Then
\begin{align*}
\trace_{K/\mathbb{Q}}((a-b)c)=\sum_{x \in \mathcal{M}(b)} \lambda_x(\trace_{K/\mathbb{Q}}(ax^2)-1)=0.
\end{align*}
Given that $\lambda_x>0$ for all $x \in \mathcal{M}(b)$, it holds that $\mathcal{M}(b) \subseteq \mathcal{M}(a)$. By a similar argument, we deduce that $\mathcal{M}(a) \subseteq \mathcal{M}(b)$, and so $\mathcal{M}(a)=\mathcal{M}(b)$. Since $\mu(a)=\mu(b)$, by the perfectness of $ax^2$, it must hold that $a=b$, which is a contradiction.
\end{proof}
Denote $\mathcal{F}_K^1=\mathcal{F}_K \cap K_{>>0}^1$. Then by the lemma above,
\begin{align*}
\vol(\mathcal{F}_K^1)=\vol\left(\bigcup_{a \in P_K} \mathcal{V}^1(a)\right)=\sum_{a \in P_K}\vol(\mathcal{V}^1(a)) \geq n_K \min_{a \in P_K} \vol(\mathcal{V}^1(a)),
\end{align*}
(here by volume, we mean the volume of the geometric object attained after canonically embedding into the space $\mathbb{R}^n$).
First, we will determine an upper bound for $\vol(\mathcal{F}_K^1)$. Let $T_n$ denote the set of elements $x$ in ${\mathbb{R}^+}^n$ such that $|x|_1 \leq 1$. Then
\begin{align*}
\vol(\mathcal{F}_K^1) \leq \vol(K_{>>0}^1) \leq \vol(T_n)=\frac{1}{n!},
\end{align*}
(since every element $a \in K_{>>0}$ will be mapped to ${\mathbb{R}^+}^n$ under the canonical embedding, by the definition of totally positive elements).

Before finding a lower bound for $\vol(\mathcal{V}^1(a))$, we will prove the following lemmas.
\begin{lemma}\label{lem1}
For all totally real, $A$-reducible fields $K$ with ring of integers $\mathcal{O}_K$ and unit group $\mathcal{O}_K^\times$ respectively,
\begin{align*}
\max_{a \in K_{>>0}}\left(\frac{\min_{u \in \mathcal{O}_K^\times}\trace_{K/\mathbb{Q}}(au^2)}{\mu(a)}\right)\leq e^{\sqrt{\eta_K^2+\theta_K}},
\end{align*}
where $\eta_K$ and $\theta_K$ are defined as in Theorem \ref{thm1}.
\end{lemma}
\begin{proof}
When $K$ is unit reducible, $\eta_K=0$ follows by definition. For non-unit reducible fields, suppose that $b \in K_{>>0}$ is such that
\begin{align*}
\max_{a \in K_{>>0}}\left(\frac{\min_{u \in \mathcal{O}_K^\times}\trace_{K/\mathbb{Q}}(au^2)}{\mu(a)}\right)=\frac{\min_{u \in \mathcal{O}_K^\times}\trace_{K/\mathbb{Q}}(bu^2)}{\mu(b)}.
\end{align*}
Suppose that $x \in \mathcal{M}(b)$, and associate to $K$ the embeddings $\sigma_1,\dots,\sigma_n$ that send $K$ to $\mathbb{R}$. Then
\begin{align*}
&\frac{\min_{u \in \mathcal{O}_K^\times}\trace_{K/\mathbb{Q}}(bu^2)}{\mu(b)}=\frac{\min_{u \in \mathcal{O}_K^\times}\trace_{K/\mathbb{Q}}\left(\frac{u^2}{x^2}bx^2\right)}{\mu(b)} \leq\frac{\min_{u \in \mathcal{O}_K^\times} \max_i \sigma_i\left(\frac{u^2}{x^2}\right)\trace_{K/\mathbb{Q}}(bx^2)}{\mu(b)}
\\&=\min_{u \in \mathcal{O}_K^\times} \max_i \sigma_i\left(\frac{u^2}{x^2}\right) = \min_{u \in \mathcal{O}_K} \max_i e^{|\log(\sigma_i(u^2/x^2))|}\leq \min_{u \in \mathcal{O}_K} e^{\sqrt{\sum_{i=1}^n(\log(|\sigma_i(u^2)|)-\log(|\sigma_i(x^2)|))^2}}.
\end{align*}
Consider the embedding $\Log: K \to \mathbb{R}^n$, $\Log(y)=(\log(|\sigma_1(y)|), \log(|\sigma_2(y)|),\dots,\log(|\sigma_n(y)|)$. Then $\Log(\mathcal{O}_K^\times)\triangleq \Lambda_K$ generates a lattice of rank $n-1$ in $\mathbb{R}^n$ (which is known as the log-unit lattice), such that, for any $(x_1,\dots,x_n) \in \Lambda_K$, we have $\sum_{i=1}^n x_i =0$. Hence, any element of $\Lambda_K$ lies in the plane $\sum_{i=1}^n x_i =0$.

Given that $K$ is $A$-reducible by assumption, $\Log(x^2)$ lies in the plane $\sum_{i=1}^n x_i = C$, where $0 \leq C \leq 2\log(A)$. Let $m$ denote the plane $\sum_{i=1}^n x_i=0$ that lies closest to $\Log(x^2)$. Then the maximum possible distance between $m$ and $\Log(x^2)$ is $\frac{2\log(A)}{\sqrt{n-1}}$, in terms of the $l_2$ norm in the space $\mathbb{R}^n$.

The distance between $m$ and any point of $\Lambda_K$, in terms of the $l_2$ norm in $\mathbb{R}^n$, cannot be greater than the covering radius of $\Lambda_K$ in $\mathbb{R}^{n-1}$, which we denote $\mu(\Lambda_K)$. Then, using Pythagoras's theorem, we get
\begin{align*}
    e^{\sqrt{\sum_{i=1}^n(\log(|\sigma_i(u^2)|)-\log(|\sigma_i(x^2)|))^2}} \leq e^{\sqrt{\mu(\Lambda_K)^2+\theta_K}}.
\end{align*}
It remains to find a bound on $\mu(\Lambda_K)$. In \cite{K2}, it was proven that for any lattice $\Lambda$ of rank $d$, $\mu(\Lambda) \leq \frac{\sqrt{d}}{2} |\det(\Lambda)|^{\frac{1}{d}}$ for $d \leq 10$. Since $|\det(\Lambda_K)|=R_K$, this yields the result stated in the lemma for $n \leq 11$.

Finally, we need to consider the case $d>11$. Denote $\lambda_i$ the $i$th successive minima of $\Lambda_K$, for $1 \leq i \leq n-1$. It is easily shown that
\begin{align}
    \mu(\Lambda_K) \leq \frac{\sqrt{n-1}}{2} \lambda_{n-1}. \label{2}
\end{align}
In \cite{K3}, it was proven that
\begin{align}
\lambda_i \geq \lambda_1 \geq \sqrt{\frac{2}{n}}\frac{1}{1000}\left(\frac{\log \log n}{\log n}\right)^3. \label{1}
\end{align}
By Minkowski's theorem, we have
\begin{align*}
 \prod_{i=1}^{n-1} \lambda_i \leq \gamma_{n-1}^{n-1} R_K,
\end{align*}
where $\gamma_n$ denotes Hemite's constant of rank $n$. By Blichfeldt \cite{B}, it is known that for any $n \geq 2$,
\begin{align*}
\gamma_n \leq \frac{2}{\pi}\Gamma\left(2+\frac{n}{2}\right)^{\frac{2}{n}},
\end{align*}
so combined with \ref{1},
\begin{align*}
\lambda_{n-1} \left(\sqrt{\frac{2}{n}}\frac{1}{1000}\left(\frac{\log \log n}{\log n}\right)^3\right)^{n-2} \leq \left(\frac{2}{\pi}\right)^{n-1}\Gamma\left(2+\frac{n-1}{2}\right)^2R_K,
\end{align*}
which yields the required result after rearranging the inequality above and using inequality \ref{2}.
\end{proof}
\begin{lemma}\label{lem2}
Suppose that $a \in K_{>>0}$ Then there exists an $a^\prime \in K_{>>0}$ such that $a \sim a^\prime$ and for all $x \in \mathcal{M}(a)$,
\begin{align*}
\trace_{K/\mathbb{Q}}(x^2) \leq e^{\sqrt{\eta_K^2+\theta_K}}\frac{4}{\pi^2} \Gamma\left(2+\frac{n}{2}\right)^{\frac{4}{n}}|\Delta_K|^{\frac{2}{n}},
\end{align*}
where $\eta_K$, $\theta_K$ are defined identically as in theorem $\ref{thm1}$.
\end{lemma}
\begin{proof}
By theorem 6 in \cite{L}, and by definition of the algebraic Hermite's constant $\gamma_K$, for any $a \in K_{>>0}$ it holds that
\begin{align*}
\mu(a) \leq \gamma_K \norm_{K/\mathbb{Q}}(a)^{\frac{1}{n}} \leq \gamma_n |\Delta_K|^{\frac{1}{n}} \norm_{K/\mathbb{Q}}(a)^{\frac{1}{n}} \leq \frac{2}{\pi} \Gamma\left(2+\frac{n}{2}\right)^{\frac{2}{n}}|\Delta_K|^{\frac{1}{n}} \norm_{K/\mathbb{Q}}(a)^{\frac{1}{n}}.
\end{align*}
Given that $a$ is totally positive, clearly $a^{-1}$ is also totally positive, and so repeating this argument, we get
\begin{align*}
\mu(a)\mu(a^{-1}) \leq \frac{4}{\pi^2} \Gamma\left(2+\frac{n}{2}\right)^{\frac{4}{n}}|\Delta_K|^{\frac{2}{n}}.
\end{align*}
Assume without loss of generality that $\mu(a)=1$, so
\begin{align*}
\mu(a^{-1}) \leq \frac{4}{\pi^2} \Gamma\left(2+\frac{n}{2}\right)^{\frac{4}{n}}|\Delta_K|^{\frac{2}{n}}.
\end{align*}
We may also assume without loss of generality that $a$ is such that
\begin{align*}
\trace_{K/\mathbb{Q}}(a^{-1}) \leq e^{\sqrt{\eta_k^2+\theta_K}}\mu(a^{-1}),
\end{align*}
as otherwise, by Lemma \ref{lem1}, there exists an $a^\prime$ such that this inequality holds and $a \sim a^\prime$. Hence,
\begin{align*}
\trace_{K/\mathbb{Q}}(a^{-1}) \leq e^{\sqrt{\eta_k^2+\theta_K}}\frac{4}{\pi^2} \Gamma\left(2+\frac{n}{2}\right)^{\frac{4}{n}}|\Delta_K|^{\frac{2}{n}}.
\end{align*}
Let $A$ denote the diagonal $n$-dimensional matrix with entries $\sigma_1(a),\sigma_2(a),\dots,\sigma_n(a)$, and denote $e_1,e_2,\dots,e_n$ the eigenvalues of $A$. Then
\begin{align*}
    \min_i e_i =\min_{\mathbf{y} \in \mathbb{R}^n}\frac{\mathbf{y}^T A \mathbf{y}}{\mathbf{y}^T \mathbf{y}} \leq \frac{\mathbf{x}^T A \mathbf{x}}{\mathbf{x}^T \mathbf{x}},
\end{align*}
where $\mathbf{x}=(\sigma_1(x),\sigma_2(x),\dots,\sigma_n(x))$, and $x \in \mathcal{M}(a)$. Note that $\trace_{K/\mathbb{Q}}(x^2)=\mathbf{x}^T\mathbf{x}$ and $\trace_{K/\mathbb{Q}}(ax^2)=\mathbf{x}^TA\mathbf{x}$, so
\begin{align*}
    &\trace_{K/\mathbb{Q}}(x^2) \leq \trace_{K/\mathbb{Q}}(ax^2) \max_i \frac{1}{e_i} \leq \trace_{K/\mathbb{Q}}(ax^2) \sum_{i=1}^n \frac{1}{e_i}=\mu(a)\trace_{K/\mathbb{Q}}(a^{-1})
    \\& =\trace_{K/\mathbb{Q}}(a^{-1}) \leq e^{\sqrt{\eta_k^2+\theta_K}}\frac{4}{\pi^2} \Gamma\left(2+\frac{n}{2}\right)^{\frac{4}{n}}|\Delta_K|^{\frac{2}{n}},
\end{align*}
as required.
\end{proof}
We are now ready to find a lower bound for $\mathcal{V}(a)$. Note that we can pick $n$ elements in $\mathcal{M}(a)$ (call this set $M_a$) such that their $\mathbb{Q}$-span is equal to $K$, and so the determinant of the polytope with edges $(\sigma_1(x^2),\dots,\sigma_n(x^2))$ is greater than or equal to $|\Delta_K|$, where $x \in M_a$. Say $M_a=\{x_1,\dots,x_n\}$. We may assume without loss of generality that $a$ is such that $\trace_{K/\mathbb{Q}}(x^2) \leq e^{\sqrt{\eta_K^2+\theta_K}}\frac{4}{\pi^2} \Gamma\left(2+\frac{n}{2}\right)^{\frac{4}{n}}|\Delta_K|^{\frac{2}{n}}$, for all $x \in \mathcal{M}(a)$, by Lemma \ref{lem2}. Then
\begin{align*}
\vol(\mathcal{V}^1(a)) \geq \vol\left(\text{conv}\left(\{0\} \cup \left\{\frac{\sigma(x_i^2)}{\trace_{K/\mathbb{Q}}(x_i^2)}\right\}_{i=1}^n\right)\right)=\frac{1}{n!}|\det(U)|,
\end{align*}
where
\begin{align*}
    U=\left\{\frac{\sigma(x_i^2)}{\trace_{K/\mathbb{Q}}(x_i^2)}\right\}_{i=1}^n,
\end{align*}
so
\begin{align*}
    &|\det(U)| =\prod_{i=1}^n \frac{1}{\trace_{K/\mathbb{Q}}(x_i^2)}\cdot |\det(\sigma(x_i^2))|_{i=1}^n \geq \frac{\Delta_K}{e^{n\sqrt{\eta_k^2+\theta_K}}\left(\frac{2}{\pi}\right)^{2n} \Gamma\left(2+\frac{n}{2}\right)^{4}|\Delta_K|^{2}}
    \\&=\frac{1}{e^{n\sqrt{\eta_k^2+\theta_K}}\left(\frac{2}{\pi}\right)^{2n} \Gamma\left(2+\frac{n}{2}\right)^{4}|\Delta_K|}.
\end{align*}
Finally, this yields
\begin{align*}
n_K &\leq \frac{\vol(\mathcal{F}_K^1}{\min_{a \in P_K}\vol(\mathcal{V}^1(a))} \leq \frac{\frac{1}{n!}}{\frac{1}{n!e^{n\sqrt{\eta_k^2+\theta_K}}\left(\frac{2}{\pi}\right)^{2n} \Gamma\left(2+\frac{n}{2}\right)^{4}|\Delta_K|}}\\&=e^{n\sqrt{\eta_K^2+\theta_K}}|\Delta_K|\left(\frac{2}{\pi}\right)^{2n}\Gamma\left(2+\frac{n}{2}\right)^4 ,
\end{align*}
which proves Theorem \ref{thm1}.
\section{Proof of Theorem \ref{thm2}}
The only result we need to prove Theorem \ref{thm2} is the following lemma.
\begin{lemma}
Let $K$ be a totally real field and suppose that $a \in K_{>>0}$. Then 
\begin{align*}
\norm_{K/\mathbb{Q}}(x) \leq n^{-\frac{n}{2}}\sqrt{|\Delta_K|}\left(\frac{2}{\pi}\right)^{\frac{n}{2}}\Gamma\left(2+\frac{n}{2}\right).
\end{align*}
\end{lemma}
\begin{proof}
Suppose that $x \in \mathcal{M}(a)$. Then, using the arithmetic-geometric inequality, the definition of the algebraic Hermite's constant and Leibak/Blichfeldt's bounds,
\begin{align*}
&\norm_{K/\mathbb{Q}}(a)\norm_{K/\mathbb{Q}}(x^2) =\norm_{K/\mathbb{Q}}(ax^2) \leq \left(\frac{1}{n}\trace_{K/\mathbb{Q}}(ax^2)\right)^n \leq \frac{1}{n^n} \left(\gamma_K \norm_{K/\mathbb{Q}}(a)^{\frac{1}{n}}\right)^n \\&\leq n^{-n} \gamma_n^n |\Delta_K| \norm_{K/\mathbb{Q}}(a) \leq n^{-\frac{n}{2}}\sqrt{|\Delta_K|}\left(\frac{2}{\pi}\right)^{\frac{n}{2}}\Gamma\left(2+\frac{n}{2}\right) \norm_{K/\mathbb{Q}}(a).
\end{align*}
The result follows by dividing through by $\norm_{K/\mathbb{Q}}(a)$.
\end{proof}
Theorem \ref{thm2} follows immediately, as we now have an upper bound for $\theta_K$.

\end{document}